\documentclass{article}

\usepackage[utf8]{inputenc}
\usepackage{amsmath,amsthm,amssymb,tikz-cd,bbm}
\usepackage{hyperref}
\usepackage{bussproofs}
\usepackage{authblk}

\theoremstyle{definition}
\newtheorem{definition}{Definition}
\newtheorem{remark}{Remark}

\theoremstyle{plain}
\newtheorem{theorem}{Theorem}
\newtheorem{corollary}{Corollary}
\newtheorem{lemma}{Lemma}
\newtheorem{proposition}{Proposition}

\newcommand{\mf}{{\bf MF}}
\newcommand{\hott}{{\bf HoTT}}
\newcommand{\mtt}{\mbox{{\bf mTT}}}
\newcommand{\mttimp}{$\mathbf{mTT}_\mathsf{imp}$}
\newcommand{\czf}{\mbox{{\bf CZF}}}
\newcommand{\emtt}{\mbox{{\bf emTT}}}
\newcommand{\emttimp}{$\mathbf{emTT}_\mathsf{imp}$}
\newcommand{\emttc}{$\mathbf{emTT}^c$}
\newcommand{\mltt}{\mbox{{\bf MLTT}}}
\newcommand{\coc}{\mbox{{\bf CC}}}
\newcommand{\cocp}{$\mathbf{CC}_\mathsf{ML}$}
\newcommand{\cocpc}{$\mathbf{CC}^c_\mathsf{ML}$}

\newcommand{\ext}{\mathsf{Dc}}

\newcommand{\eqclass}[1]{\,\raisebox{2.5pt}{$\ulcorner$} #1 \raisebox{2.5pt}{$\urcorner$}}
\newcommand{\sem}[1]{[\![\, #1 \,]\!]}
\newcommand{\decorate}[1]{#1}

\title{Equiconsistency of the Minimalist Foundation \\ with its classical version}

\author{Maria Emilia Maietti\footnote{maietti@math.unipd.it} \quad Pietro Sabelli\footnote{sabelli@math.unipd.it}}

\date{Dipartimento di Matematica ``Tullio Levi-Civita''\\ University of Padua, Italy \\[2ex] \today}

\begin{document}

\maketitle

\begin{abstract}
The Minimalist Foundation, for short \mf, was conceived by the first author with G. Sambin in 2005, and fully formalized in 2009, as a common core among the most relevant constructive and classical foundations for mathematics.

To better accomplish its minimality, \mf\ was designed as a two-level type theory, with 
an intensional level \mtt, an extensional one \emtt, and an interpretation
of the latter into the first.

Here, we first show that the two levels of \mf\ are indeed equiconsistent by interpreting \mtt\ into \emtt.  Then, we show that the classical extension \emttc\ is equiconsistent with
\emtt\  by suitably extending the Gödel-Gentzen double-negation translation of classical logic in the intuitionistic one. As a consequence, \mf\ turns out to be compatible with classical predicative mathematics à la Weyl, 
contrary to the most relevant foundations for constructive mathematics.

Finally, we show that the  chain of equiconsistency results for \mf\ can be straightforwardly extended to its impredicative
version to deduce that Coquand-Huet's Calculus of Constructions equipped with basic inductive types is equiconsistent with its extensional and classical versions too.

%A relevant consequence of these results for \mf\ is that it turns out that \mf, contrary to the most relevant foundations for constructive mathematics,  is compatible with classical predicative mathematics à la Weyl. Secondly, to establish the exact proof-theoretic strength of \mf, which is still an open problem, we are no longer bound to refer to \mtt\ but we can interchangeably use also \emtt\ or \emttc.

%On the other, that, contrary to the most relevant foundations for constructive mathematics, \mf\ is compatible with classical predicative mathematics such as Aczel's \czf, Martin \mltt\ or \hott\ which become impredicative when their logic is extended with the law of excluded middle.

\end{abstract}

\section{Introduction}
This paper is a contribution to the field of foundations of mathematics, and in particular to the meta-mathematical properties of the Minimalist Foundation and of the Calculus of Constructions with basic inductive types.

It is well known that for constructive mathematics there is a wide variety of foundations, formulated not only in axiomatic set theory, such as Aczel's Constructive Zermelo-Fraenkel set theory $\mathbf{CZF}$ \cite{czf}; but also in type theory, such as Coquand-Paulin's Calculus of Inductive Constructions $\mathbf{CIC}$ \cite{CIC}, Martin-Löf's type theory $\mathbf{MLTT}$ \cite{MLTT}, and the more recent Homotopy Type Theory \hott\ \cite{hottbook}; and in category theory, such as the internal language of categorical universes like topoi \cite{lambek-scott, modular}.

The Minimalist Foundation, for short \mf, was first conceived in \cite{mtt}, and then fully formalized in \cite{m09}, to serve as a common core compatible with the most relevant constructive and classical foundations, including the mentioned ones. To meet its \textit{raison d'être}, \mf\ has been designed as a two-level foundation, consisting of an extensional level \emtt, understood as the actual theory in which constructive mathematics is formalized and developed, and an intensional level \mtt, which acts as a functional programming language enjoying a realizability interpretation à la Kleene \cite{MMM, IMMS}. 

Both levels of \mf\ are formulated as dependent type theories à la Martin-Löf enriched with a primitive notion of proposition and related form of comprehension. A remarkable difference is that  \emtt\ is
equipped with the extensional constructions of \cite{ML84} enriched with quotients and power collections of sets, while \mtt\ is equipped with the intensional types of \cite{MLTT} enriched with the collection of predicates on a set.
An interpretation of \emtt\ in \mtt\ was given in \cite{m09} using a quotient model of setoids, whose peculiar properties concerning analogous models on Martin-Löf's type theory had been analysed categorically in \cite{qu12,elqu, unifying} in terms of completions of Lawvere doctrines.

Our aim here is to show that \mf\ is equiconsistent with
its classical counterpart.
This is not true for \czf, \mltt, or \hott\ since they become impredicative when the Law of Excluded Middle is added to their intended underlying logic.

To meet our purpose, we first show that \mtt\ and \emtt\ are equiconsistent, namely that we can invert the structure of \mf\ by interpreting 
\mtt\ within \emtt. This goal is achieved using in particular the technique of canonical isomorphisms, already employed to interpret \emtt\ within \mtt\  in \cite{m09} and \emtt\ within \hott\ in \cite{MFhott}. Then, to simplify our goal we refer to the {\it classical counterpart} of \mf\ by just considering the extension \emttc\ of the extensional level \emtt\ with the Law of Excluded Middle. The equiconsistency of the two levels of \mf\ justifies this choice. 

The next step of our work is to show that \emttc\ is equiconsistent with \emtt. To do so, we adapt the Gödel-Gentzen double-negation translation  of classical logic in the intuitionistic one (see for example \cite{Troelstra}) to interpret \emttc\ within \emtt, exploiting in particular the fact that the type constructors of \emtt\ preserve the $\neg\neg$-stability of their propositional equalities. The proof indeed requires more care than for systems of many-sorted logic, such as Heyting arithmetic with finite types in \cite{Troelstra}, because of the interaction in \emtt\ between propositions and collections.

As a consequence of the equiconsistency of \emtt\ with \emttc, we show that real numbers à la Dedekind do not form a set neither in \emtt, nor in \emttc. 
Therefore, \emttc\ can be taken as a foundation of 
classical predicative mathematics in the spirit of Weyl in \cite{DasKontinuum}, and of course \mf\ through \emtt\ becomes compatible with it. Another benefit of our equiconsistency results is that, to establish the exact proof-theoretic strength of \mf, which is still an open problem, we are no longer bound to refer to \mtt\ but we can interchangeably use also \emtt\ or \emttc.

Finally, by exploiting the fact that the intensional level \mtt\ is a {\it predicative version} of Coquand-Huet's Calculus of Constructions \cite{CoC}, we show that
the chain of equiconsistency results for \mf\ can be straightforwardly adapted to an impredicative version of \mf\ whose intensional level is the Calculus of Constructions equipped with inductive types from the first-order fragment of \mltt, which we call \cocp, thus extending the result in \cite{CoCprooftheory} on the equiconsistency of the logical base of the calculus with its classical version without relying on normalization properties of \cocp.

A related relevant goal would be to investigate how to extend the equiconsistency results presented here to extensions of the Minimalist Foundation, and its impredicative version based on \cocp, with the inductive and coinductive definitions investigated in \cite{mmr21,mmr22,MFwtypes}.
While the equiconsistency proof of \mtt\ with \emtt\ can be transferred smoothly to these extensions, this does not apply to the Gödel-Gentzen double-negation translation and this goal is left to future work.

\section{Brief recap of the Minimalist Foundation}

In this section, we recall the fundamental facts about the Minimalist Foundation, together with some useful conventions to work with it.

The name \textit{Minimalist Foundation}, abbreviated as \mf, refers to the two-level system introduced in \cite{m09} following the requirements in \cite{mtt}, consisting of an {\it extensional level}  $\mathbf{emTT}$ for {\it extensional minimal Type Theory}, an {\it intensional level} $\mathbf{mTT}$ for {\it minimal Type Theory}, and an interpretation of the first in the latter.
Both \emtt\ and \mtt\ are formulated as dependent type theories  with four kinds of types: small propositions, propositions, sets, and collections (denoted respectively $prop_s$, $prop$, $set$ and $col$). Sets are particular collections, just as small propositions are particular propositions. Moreover, we identify a proposition (respectively, a small proposition) with the collection (respectively, the set) of its proofs. Eventually, we have the following square of inclusions between kinds.
\[
% https://tikzcd.yichuanshen.de/#N4Igdg9gJgpgziAXAbVABwnAlgFyxMJZARgBpiBdUkANwEMAbAVxiRAB12BbOnACzgAzYAGMIDAL4gJpdJlz5CKMgAYqtRizace-IcDgwcUmXOx4CRFeXX1mrRB268BwtACcIaE7JAZzilakatR2Wo46LvoeXgD6cCbqMFAA5vBEoIKeXEgAzNQ4EEgATAV0WAxsfBAQANbSvlkQOYhkIIVI1u3llY7VdQ2Z2SUFRYhdOD1VNfWmIE0t+e1jbZMV0wMSFBJAA
\begin{tikzcd}[row sep=small, column sep = small]
prop_s \arrow[d, hook] \arrow[r, hook] & set \arrow[d, hook] \\
prop \arrow[r, hook]                   & col                
\end{tikzcd}
\]
This fourfold distinction allows one to differentiate, on the one hand, between logical and mathematical entities. On the other, between inductively generated domains and open-ended domains (corresponding to the usual distinction between sets and classes in set theory – or the one between small and large types in Martin-Löf's type theory with a universe), thus guaranteeing the predicativity of the theory.

In both levels, propositions are those of predicate logic with equality; a proposition is small if all its quantifiers and propositional equalities are over sets. The base sets include the empty set $\mathsf{N_0}$, the singleton set $\mathsf{N_1}$, the set constructors are the dependent sum $\Sigma$, the dependent product $\Pi$, the disjoint sum $+$, the list constructor $\mathsf{List}$. What differentiates the set constructors of the two levels is the presence, only at the extensional level \emtt, of a constructor $A/R$ to quotient a set $A$ by a small equivalence relation $R$ depending on the product $A\times A$. Regarding collections, while \mtt\ is equipped with a universe of small propositions $\mathsf{Prop_s}$ and function spaces $A\rightarrow \mathsf{Prop_s}$ from a set $A$ towards $\mathsf{Prop_s}$,
collections  of \emtt\  include a classifier $\mathcal{P}(1)$ of small propositions \textit{up to equiprovability}, which is often called the \textit{power collection of the singleton}, and power collections $\mathcal{P}(A)$ for each set $A$, defined as the function space $\mathcal{P}(A) :\equiv A\rightarrow \mathcal{P}(1)$.
It is important to notice that, contrary to its definition à la Russell in \cite{m09}, the universe $\mathsf{Prop_s}$ of small propositions of $\mathbf{mTT}$ is presented here à la Tarsky through the following rules.
\[
 \textsf{F-$\mathsf{Prop_s}$} \;
\frac
{}
{\mathsf{Prop_s} \; col}
\qquad
 \textsf{I-$\mathsf{Prop_s}$} \;
\frac
{
\varphi \; prop_s
}
{
\widehat{\varphi} \in \mathsf{Prop_s}
}
\qquad
 \textsf{E-$\mathsf{Prop_s}$} \;
\frac
{
c \in \mathsf{Prop_s}
}
{
\mathsf{T}(c) \; prop_s
}
\]
\[
 \textsf{C-$\mathsf{Prop_s}$} \;
\frac
{
\varphi \; prop_s
}
{\mathsf{T}(\widehat{\varphi}) = \varphi \; prop_s}
\qquad
 \textsf{$\eta$-$\mathsf{Prop_s}$} \;
\frac
{
c \in \mathsf{Prop_s}
}
{\widehat{\mathsf{T}(c)} = c \in \mathsf{Prop_s}}
\]
\[
\textsf{Eq-$\mathsf{Prop_s}$} \;
\frac
{
\varphi = \psi \; prop_s
}
{
\widehat{\varphi} = \widehat{\psi} \in \mathsf{Prop_s}
}
\qquad
 \textsf{Eq-E-$\mathsf{Prop_s}$} \;
\frac
{
c = d \in \mathsf{Prop_s}
}
{\mathsf{T}(c) = \mathsf{T}(d) \; prop_s}
\]
Finally, in both levels collections are closed under dependent sums $\Sigma$.

The intensionality of \mtt\ means that propositions are proof-relevant, the propositional equality is intensional à la Leibniz, and the only computation rules are $\beta$-equalities. Conversely, the extensionality of \emtt\ means that the propositional equality reflects judgemental equality, all $\eta$-equalities are valid, and propositions are proof-irrelevant; in particular, in \emtt\ there is a canonical proof-term $\mathsf{true}$ for propositions, and sometimes we render the judgement $\mathsf{true} \in \varphi \; [\Gamma]$ as $\varphi \; \mathsf{true} \; [\Gamma]$ to enhance its readability.

The two levels of \mf\ are related by an interpretation of \emtt\ into \mtt\ in \cite{m09} using a quotient model, analyzed categorically in \cite{qu12,elqu}, together with canonical isomorphisms as in \cite{hofmann} but without any use of choice principles in the meta-theory. This interpretation shows that the link between  the two levels of \mf\ fulfils  Sambin's forget-restore principle in \cite{toolbox}, saying that that computational information present only implicitly in the derivations of the extensional level can be restored as terms of the intensional level, from which, in turn, programs can be extracted as shown by the realizability model in \cite{IMMS}. 

The same process described above can be performed for the two-level extension of \mf\ with inductive and coinductive topological definitions in \cite{mmr21,mmr22}, which amount to include all inductive and coinductive predicate definitions as shown in \cite{topcount,coinductivecounterpart}.

Furthermore, the two-level structure of \mf\ can be extended to its impredicative version
described in Section \ref{impmf}, by exploiting the fact that \mtt\ is indeed a {\it predicative version} of
Coquand-Huet's Calculus of Constructions in \cite{CoC,metaCoC}.

\subsection{Notation}

In dependent types theory, where the  definitions of language and derivability are intertwined, the expressions of the calculus have to be introduced first with a so-called \textit{pre-syntax} (see \cite{streicher}); the pre-syntaxes of both levels consist of four kinds of entities: pre-contexts, pre-types, pre-propositions, and pre-terms; we assume that the pre-syntax is fully annotated, in the sense that each (pre-)term has all the information needed to infer the (pre-)type it belongs to, although for readability, we will leave a lot of that implicit in the following.

We use the entailment symbol
$\mathcal{T} \vdash \mathcal{J}$ to express that {\it the theory $\mathcal{T}$ derives the judgement $\mathcal{J}$}. Moreover, when doing calculations or writing inference rules, we will often follow the usual conventions of omitting the piece of context common to all the judgements involved; furthermore, the placeholder $type$ in a judgement of the form $A \; type \; [\Gamma]$ stands for one of the four kinds $prop_s$, $prop$, $set$ or $col$, always with the same choice if it has multiple occurrences in the same sentence or inference rule.

% When dealing with the syntax of a dependent type theory, we distinguish three different notions of equality: \textit{meta-equality} between entities in our meta-theory $e_1 \equiv e_2$; \textit{judgemental equality} between types $A = B$ and terms $a = b \in A$, with the latter sometimes shortened as $a = b$ when the common type is clear from the context; and \textit{propositional equality} type between two terms of a given type, which is written $\mathsf{Id}(A,a,b)$ in the intensional level \mtt, $\mathsf{Eq}(A,a,b)$ in the extensional one \emtt, and in the former case is often abbreviated as $a =_A b$.

We will make use of the following common shorthands: the propositional equality predicate $\mathsf{Eq}(A,a,b)$ of \emtt\ will be often abbreviated as $a =_A b$; we will often write $f(a)$ as a shorthand for $\mathsf{Ap}(f,a)$; we reserve the arrow symbol $\to$ \textit{(resp. $\times$)} as a shorthand for a non-dependent product \textit{(resp. for non-dependent product sets)}, while we denote the implication connective with the arrow symbol $\Rightarrow$; the projections from a dependent sum are defined as $\pi_{\mathsf{1}}(z) :\equiv \mathsf{El}_{\Sigma}(z,(x,y).x)$ and
$\pi_{\mathsf{2}}(z) :\equiv \mathsf{El}_{\Sigma}(z,(x,y).y)$; negation, the true constant, and logical equivalence are defined respectively as
$\neg \varphi :\equiv \varphi \Rightarrow \bot$,
$\top :\equiv \neg\bot$ and $\varphi \Leftrightarrow \psi :\equiv \varphi \Rightarrow \psi \wedge \psi \Rightarrow \varphi$.

In $\mathbf{emTT}$, we define the decoding of a term $U \in \mathcal{P}(1)$ as the small proposition $\ext(U) :\equiv \mathsf{Eq}(\mathcal{P}(1),U,[\top]) \; prop_s$, observing that it satisfies the following computation rule.
\[
 \textsf{C-$\mathcal{P}(1)$} \;
\frac
{\varphi \; prop_s}
{\ext([\varphi]) \Leftrightarrow \varphi \; \mathsf{true}}
\]
As usual in \emtt, we let $\mathcal{P}(A) :\equiv A \to \mathcal{P}(1)$ denote the power collection of a set $A$, and $a \,\varepsilon\, V :\equiv \mathsf{Ap}(V,a)$ denote the (propositional) relation of membership between terms $a \in A$ and subsets $V \in \mathcal{P}(A)$; accordingly, we will employ the common set-builder notation $\{x \in A \,|\, \varphi(x)\} :\equiv (\lambda x \in A)[\varphi(x)]$ for defining a subset by comprehension through a small predicate.

\section{Equiconsistency of the two levels of \mf}\label{eqlevels}
The presence of an intensional and an extensional level in the Minimalist Foundation resembles very closely the two versions, intensional and extensional, of Martin-Löf's type theory.
Indeed, both \emtt\ and \mtt\ are formulated as dependent type theories extending versions of Martin-Löf's type theory enriched with a primitive notion of proposition; moreover, propositions are thought of as types to guarantee in particular their comprehension. More precisely, \emtt\
extends the first-order version in \cite{ML84} with quotients and  power collections of sets,
while \mtt\ extends the first-order version in \cite{MLTT} with
the collection of predicates on a set.

While it is notoriously difficult to interpret the extensional version of Martin-Löf's type theory in \cite{ML84} into its intensional one in \cite{MLTT}, especially in the presence of universes (see for example \cite{hofmann, palmgren}), in the other direction the task is trivial. Indeed, the extensional version is a direct extension of the intensional one obtained mainly by strengthening the elimination rule of the identity type to make it reflect judgemental equality.

In the case of the Minimalist Foundation, an interpretation of the extensional level into the intensional one was given in \cite{m09}, in which the theory \emtt\ is interpreted in a quotient model of so-called setoids constructed over the theory \mtt. However, contrary to Martin-Löf's type theory, \mtt\ is not an extension of \emtt\ because of the discrepancy between the intensional universe of small propositional in \mtt\ and the power collection of the singleton in \emtt. The question of whether \mtt\ can be interpreted in \emtt\ is therefore not trivial, and it is what we are going to answer positively in this section.

We first observe that, to fix the discrepancy described above it is sufficient to add an axiom $\textsf{propext}$ of \textit{propositional extensionality} to \emtt; in this way an interpretation of \mtt\ into $\mathbf{emTT}+\textsf{propext}$ is easily achieved; then, we can interpret $\mathbf{emTT}+\textsf{propext}$ back into \emtt\ by employing the technique of {\it canonical isomorphisms}, already used in the interpretation of \emtt\ in \mtt\ in \cite{m09}, independently adopted for interpretations in other type-theoretic systems in \cite{hofmann, spadetto}, and later employed also in \cite{MFhott} to show the compatibility of \emtt\ with \hott, given that propositional extensionality is just an equivalent presentation in \emtt\ of propositional univalence. Moreover, such interpretation of $\mathbf{emTT}+\textsf{propext}$ into \emtt\ will be \textit{effective}, in the sense of \cite{m09} and \cite{eliminating}, since its Validity Theorem \ref{validitypropext} can be constructively implemented as a translation of derivations of the source theory into derivations of the target theory. Finally, as a byproduct we will also conclude that $\mathbf{emTT}+\textsf{propext}$ is conservative
over \emtt.

\subsection{Interpreting \mtt\ into \texorpdfstring{$\mathbf{emTT}+\mathsf{propext}$}{emTT+propext}}
Recall that the power collection  $\mathcal{P}(1)$ of the singleton considers propositions up to equiprovability, while, in the intensional case, the universe $\mathsf{Prop_s}$ of small propositions does not; in particular, it is clear that $\mathcal{P}(1)$ cannot interpret $\mathsf{Prop_s}$ since the computation rule of the former on the left is weaker than that of the latter on the right:
\[
\textsf{C-$\mathcal{P}(1)$} \;
\frac
{\varphi \; prop_s}
{\ext([\varphi]) \Leftrightarrow \varphi \; \mathsf{true}}
\qquad\qquad 
 \textsf{C-$\mathsf{Prop_s}$} \;
\frac
{\varphi \; prop_s}
{\mathsf{T}(\widehat{\varphi}) = \varphi \; prop_s}
\]
To rectify this situation, we consider adding to $\mathbf{emTT}$ axioms for propositional extensionality.
\[
\textsf{propext}\;\frac
{\varphi \; prop \quad \psi \; prop \quad \varphi \Leftrightarrow \psi \; \mathsf{true}}
{\varphi = \psi \; prop}
\]
\[
\mathsf{prop_{s}ext}\;\frac
{\varphi \; prop_s \quad \psi \; prop_s \quad \varphi \Leftrightarrow \psi \; \mathsf{true}}
{\varphi = \psi \; prop_s}
\]
Identifying equality and equiprovability for propositions clearly fixes the discrepancy between the two collections $\mathcal{P}(1)$ and $\mathsf{Prop_s}$. The resulting theory will be called $\mathbf{emTT}+\textsf{propext}$. As the next proposition shows, it can easily interpret $\mathbf{mTT}$.

\begin{proposition}\label{eqc1}
$\mathbf{mTT}$ is interpretable in $\mathbf{emTT}+\mathsf{propext}$.
\end{proposition}
\begin{proof}
Consider the translation that renames in the pre-syntax $\mathsf{Id}$ to $\mathsf{Eq}$, $\mathsf{Prop_s}$ to $\mathcal{P}(1)$, $\widehat{-}$ to $[-]$, $\mathsf{T}(-)$ to $\ext(-)$, and all the proof-term constructors $\mathsf{El}_\bot$, $\lambda_\Rightarrow$, $\mathsf{Ap}_\Rightarrow$, $\langle-,_\wedge-\rangle$, $\pi_1^\wedge$, $\pi_2^\wedge$, $\mathsf{inl}_\vee$,$\mathsf{inr}_\vee$, $\mathsf{El}_\vee$, $\lambda_\forall$, $\mathsf{Ap}_\forall$, $\langle-,_\exists-\rangle$, $\mathsf{El}_\exists$, $\mathsf{id}$, $\mathsf{El}_\mathsf{Id}$ to the canonical proof-term $\mathsf{true}$.

It is straightforward to check that the above translation is an interpretation of $\mathbf{mTT}$ into $\mathbf{emTT}+\mathsf{propext}$. In particular, in \emtt\ the rule $\textsf{prop-mono}$ collapses all the proof-terms to the canonical one $\mathsf{true}$, while the computation rule $\textsf{C-Prop\textsubscript{s}}$ of the universe is satisfied thanks to the additional axiom $\textsf{propext}$.
\end{proof}

\iffalse
Additionally, it is easy to see that the quotient model used to interpret $\mathbf{emTT}$ in $\mathbf{mTT}$ actually interprets also $\mathbf{emTT}+\mathsf{propext}$.

\begin{proposition}\label{eqc2}
$\mathbf{emTT}+\mathsf{propext}$ can be interpreted in the quotient model of $\mathbf{mTT}$. 
\end{proposition}
\begin{proof}
We just need to check the additional rules of propositional extensionality. In the quotient model, equality of $\mathbf{emTT}$-collections are interpreted as the existence of a so-called \textit{canonical isomorphism} between the $\mathbf{mTT}$-\textit{extensional collections} interpreting them; however, in the case of extensional collections interpreting $\mathbf{emTT}$-propositions, the existence of a canonical isomorphism amounts exactly to their equiprovability.
\end{proof}
\fi

We now turn to the task of interpreting $\mathbf{emTT}+\textsf{propext}$ into $\mathbf{emTT}$. The key idea is to interpret a proposition of $\mathbf{emTT}+\textsf{propext}$ as a proposition of $\mathbf{emTT}$ \textit{up to equiprovability}, that is as an equivalence class of logically equivalent $\mathbf{emTT}$-propositions. Since collections may depend on propositions, crucially thanks to the $\textsf{prop-into-col}$ rule and the quotient set constructor, we will have to extend this rationale to all types of \emtt\  by interpreting them as types \textit{up to equivalent logical components}.

\subsection{Canonical isomorphisms}
In this subsection, we define a notion of {\it canonical isomorphism} in \emtt\ that will be used to interpret $\mathbf{emTT}+\textsf{propext}$.
As already mentioned, the idea of using canonical isomorphisms to interpret extensional equalities in type theory was originally conceived in \cite{m09} between objects of a quotient model, and, independently by Hofmann in \cite{hofmann}, whilst with the additional help of the Axiom of Choice in the meta-theory. The results in \cite{hofmann} were later made effective in \cite{Oury05,eliminating} with the adoption of a heterogeneous equality. Finally, canonical isomorphisms were used recently also in \cite{MFhott}, and the treatment given here closely resembles it both in definitions and proof techniques. Since in this subsection our only object theory will be $\mathbf{emTT}$, we assume that all the judgements are meant to be judgements derivable in $\mathbf{emTT}$.

Recall that a \textit{functional term} from a collection $A$ to another collection $B$ defined over the same context is a term $t(x) \in B \; [x \in A]$ of type $B$ defined in the context extended by $A$. We say that $t$ is an \textit{isomorphism} if there exists another functional term $t^{-1}(y) \in A \; [y \in B]$ from $B$ to $A$ such that
\[
(\forall x \in A)t^{-1}(t(x)) =_A x \;\wedge\; (\forall y \in B)t(t^{-1}(y)) =_B y
\]
and in that case, we say that the collections $A$ and $B$ are \textit{isomorphic}.

\begin{definition}[Canonical isomorphisms]
We inductively define a family of functional terms, called \textit{canonical isomorphisms}, between collections depending on a context (which, as customary, in each of the following clauses will be left implicit):

\begin{enumerate}
\item if $\varphi$ and $\psi$ are logically equivalent propositions (that is, if $\varphi \Leftrightarrow \psi \; \mathsf{true}$ is derivable), then the unique functional term $\mathsf{true} \in \psi \; [x \in \varphi]$ is a canonical isomorphism;
\item the identities of the base types $\mathsf{N_0}$, $\mathsf{N_1}$, and $\mathcal{P}(1)$ are canonical isomorphisms;
\item if $\tau(x) \in B \; [x \in A]$ is a canonical isomorphism between dependent sets, then the functional term
\[
(a_1,\dots,a_n) \in \mathsf{List}(A) \mapsto (\tau(a_1),\dots,\tau(a_n)) \in \mathsf{List}(B)
\]
extending $\tau(x)$ to lists element-wise is a canonical isomorphism; it can be formally defined as
\[
\mathsf{El_{List}}(l,\epsilon,(x,y,z).\mathsf{cons}(z,\tau(x))) \in \mathsf{List}(B) \; [l \in \mathsf{List}(A)]
\]
\item if $\tau(x) \in A' \; [x \in A]$ and $\sigma(x) \in B' \; [x \in B]$ are two canonical isomorphisms between dependent sets, then their coproduct
\begin{align*}
\mathsf{inl}(a) \in A+B & \mapsto \mathsf{inl}(\tau(a)) \in A'+B' \\
\mathsf{inr}(b) \in A+B & \mapsto \mathsf{inr}(\sigma(b)) \in A'+B'
\end{align*}
is a canonical isomorphism; it can be formally defined as
\[
\mathsf{El_{+}}(z,(x).\tau(x),(y).\sigma(y)) \in A'+B' \; [z \in A+B] 
\]
\item if $B(x) \; col \; [x \in A]$ and $B'(x) \; col \; [x \in A']$ are two dependent collections, and there are canonical isomorphisms
\[
\tau(x) \in A' \; [x \in A]
\qquad
\sigma(x,y) \in B'(\tau(x)) \; [x \in A, y \in B(x)]
\]
then the functional term
\[
\langle a, b \rangle \in (\Sigma x \in A)B(x) \mapsto \langle \tau(a), \sigma(a,b) \rangle \in (\Sigma x \in A')B'(x)
\]
is a canonical isomorphism; it can be formally defined as
\[
\mathsf{El}_\Sigma(z,(x,y).\langle \tau(x), \sigma(x,y) \rangle) \in (\Sigma x \in A')B'(x) \; [z \in (\Sigma x \in A)B(x) ]
\]
\item if $B(x) \; col \; [x \in A]$ and $B'(x) \; col \; [x \in A']$ are two dependent collections such that their dependent product is a collection, and there are canonical isomorphisms
\[
\tau(x) \in A \; [x \in A']
\qquad
\sigma(x,y) \in B'(x) \; [x \in A', y \in B(\tau(x))]
\]
then the following is a canonical isomorphism
\[
(\lambda x \in A')\sigma(x,\mathsf{Ap}(f,\tau(x))) \in (\Pi x \in A')B'(x) \; [f \in (\Pi x \in A)B(x)]
\]
\item if $\tau(x) \in B \; [x \in A]$ is a canonical isomorphism between sets, $R(x,y)$ is a small equivalence relation on $A$, and $S(x,y)$ is a small equivalence relation on $B$ such that $R(x,y) \Leftrightarrow S(\tau(x),\tau(y)) \; \mathsf{true} \; [x,y \in A]$ holds, then the functional term
\[
[a] \in A/R \mapsto [\tau(a)] \in B/S
\]
obtained by passing $\tau(x)$ to the quotient is a canonical isomorphism; it can be formally defined as
\[
\mathsf{El_Q}(z,(x).[\tau(x)]) \in B/S \; [z \in A/R]
\]
\end{enumerate}
\end{definition}

We now derive some fundamental properties about canonical isomorphisms.

Recall that a telescopic substitutions $\gamma$ from a context $\Delta$ to a context $\Gamma \equiv x_1 \in A_1,\dots,x_n\in A_n$ is a list of $n$ terms
\[
\gamma \equiv t_1 \in A_1 \; [\Delta]\,,\,\cdots,\,
t_n \in A_n[t_1/x_1]\cdots[t_{n-1}/x_{n-1}] \; [\Delta]
\]
We write it as the derived judgement $\gamma \in \Gamma \; [\Delta]$. Moreover, if $B \; type \; [\Gamma]$, we write $B[\gamma] \; type \; [\Delta]$ for the substituted type $B[t_1/x_1]\cdots[t_n/x_n]$, and analogously for terms.

\begin{lemma}\label{cansub}
If $\tau \in B \; [\Gamma, x \in A]$ is a canonical isomorphism and $\gamma \in \Gamma \; [\Delta]$ is a telescopic substitution, then also $\tau[\gamma,x] \in B[\gamma] \; [\Delta, x \in A[\gamma]]$ is a canonical isomorphism.
\end{lemma}
\begin{proof}
By induction on the definition of canonical isomorphism.
\end{proof}

\begin{proposition}\label{can}
Canonical isomorphisms enjoy the following properties:
\begin{enumerate}
\item identities are canonical isomorphisms;
\item canonical isomorphisms are indeed isomorphisms, and their inverses are again canonical isomorphisms;
\item the composition of two (composable) canonical isomorphisms is a canonical isomorphism;
\item there exists at most one canonical isomorphism between each pair of collections.
\end{enumerate}
\end{proposition}
\begin{proof}
The proof is analogous to the one performed for \hott\ in Proposition 4.11 of \cite{MFhott}.
%\begin{enumerate}

\textit{Point $1$} follows by induction on the construction of the collection, exploiting the $\eta$-equalities of the corresponding constructors.

\textit{Point $2$} follows by induction on the definition of canonical isomorphism; in particular, thanks to the fact that $\Leftrightarrow$ is symmetric in the case of propositions. We spell out the case of dependent products. By induction hypothesis, there exist the two canonical inverses
\[
\tau^{-1}(x) \in A' \; [x \in A] \qquad \sigma^{-1}(x,y) \in B(\tau(x)) \; [x \in A',y \in B'(x)]
\]
if we substitute the second by the first we obtain
\[
\sigma^{-1}(\tau^{-1}(x),y) \in B(\tau(\tau^{-1}(x))) = B(x) \; [x \in A, y \in B'(\tau^{-1}(x))]
\]
which is again canonical thanks to Lemma \ref{cansub}, so that we can consider the canonical 
isomorphism
\[
(\lambda x \in A)\sigma^{-1}(\tau^{-1}(x),f(\tau^{-1}(x))) \in (\Pi x \in A)B(x) \; [f \in (\Pi x \in A')B'(x)]
\]
which can be easily checked to be the inverse.

For \textit{point 3}, observe that two objects are related by a canonical isomorphism only if they have the same outermost constructor or if they are both propositions; in the latter case, we rely on the transitivity of $\Leftrightarrow$; in the former case, we proceed by induction on the outermost constructor. We spell out the case of the dependent product. Suppose to have the following canonical isomorphisms
\begin{align*}
& \tau(x) \in A \; [x \in A'] \\
& \tau'(x) \in A' \; [x \in A''] \\
& \sigma(x,y) \in B'(x) \; [x \in A', y \in B(\tau(x))] \\
& \sigma'(x,y) \in B''(x)  \; [x \in A'', y \in B'(\tau'(x))]
\end{align*}
By Lemma \ref{cansub} also the following substituted morphism is canonical
\[
\sigma(\tau'(x),y) \in B'(\tau'(x)) \;  \; [x \in A'', y \in B(\tau(\tau'(x)))]
\]
By inductive hypothesis, the following isomorphisms obtained by composition are canonical
\[
\tau(\tau'(x)) \in A \; [x \in A''] \quad \sigma'(x,\sigma(\tau'(x),y)) \in B''(x) \; [x \in A'', y \in B(\tau(\tau'(x)))]
\]
We must check that the composition of the two canonical isomorphisms 
\begin{align*}
(\lambda x \in A')\sigma(x,\mathsf{Ap}(f,\tau(x))) & \in (\Pi x \in A')B'(x) \; [f \in (\Pi x \in A)B(x)] \\
(\lambda x \in A')\sigma'(x,\mathsf{Ap}(f,\tau'(x))) & \in (\Pi x \in A'')B''(x) \; [f \in (\Pi x \in A')B'(x)]
\end{align*}
is canonical, but their composition is equal to
\begin{align*}
& (\lambda x \in A'')\sigma'(x,\mathsf{Ap}( (\lambda x \in A')\sigma(x,\mathsf{Ap}(f,\tau(x))) ,\tau'(x))) = \\
& (\lambda x \in A'')\sigma'(x,\sigma(\tau'(x),\mathsf{Ap}(f,\tau(\tau'(x)))))
\end{align*}
which is canonical by definition.

\textit{Point 4} is trivial in the case of propositions; in the other cases, it is proven by induction on the outermost constructor of the two collections.
\end{proof}

We can extend the notion of canonical isomorphisms to contexts of \emtt.

\begin{definition}
We inductively define a family of
telescopic substitutions between contexts, called again \textit{canonical isomorphisms}:
\begin{itemize}
\item the empty telescopic substitution between empty contexts $() \in () \; [()]$ is a canonical isomorphism;
\item if $A \; col \; [\Gamma]$ and $B \; col \; [\Delta]$ are two dependent collections, $\sigma \in \Delta \; [\Gamma]$ is a canonical isomorphism between contexts, and $\tau \in B[\sigma] \; [\Gamma, x \in A]$ is a canonical isomorphism between collections, then the extension $\sigma, \tau \in (\Delta, x \in B) \; [\Gamma, x \in A]$ is a canonical isomorphism.
\end{itemize}
\end{definition}
It is easy to check by induction that canonical isomorphisms between contexts inherit the properties of Proposition \ref{can}.

\begin{definition}\label{related}
We say that \textit{two contexts $\Gamma$ and $\Delta$ are canonically related} if there exists a (necessarily unique) canonical isomorphism between them.

We say that \textit{two dependent types $A \; type\; [\Gamma]$ and $B \; type\; [\Delta]$ are canonically related} if their contexts are canonically related, and there exists a (necessarily unique) canonical isomorphism between $A$ and $B[\sigma]$, where $\sigma \in \Delta \; [\Gamma]$ is the canonical isomorphism between contexts.
%Given an equivalence class $\mathbf{\Gamma}$ of canonically related contexts, we say that two dependent collections $A \; col\; [\Gamma]$ and $B \; col\; [\Delta]$ with $\ulcorner \Gamma \; ctx \urcorner \equiv \ulcorner \Delta \; ctx \urcorner \equiv \mathbf{\Gamma}$ are \textit{canonically related over $\mathbf{\Gamma}$} if there exists a (necessarily unique) canonical isomorphism between $A$ and $B[\sigma]$ in the category of dependent collections over $\Gamma$, where $\sigma \in \Delta \; [\Gamma]$ is the canonical isomorphism in the category of contexts.

Finally, we say that \textit{two telescopic substitutions $\gamma \in \Gamma \; [\Gamma']$ and $\delta \in \Delta \; [\Delta']$ are canonically related} if both their domain and codomain are canonically related and the compositions of telescopic substitutions depicted pictorially in the following square are judgementally equal
\[
% https://tikzcd.yichuanshen.de/#N4Igdg9gJgpgziAXAbVABwnAlgFyxMJZABgBoBGAXVJADcBDAGwFcYkQAdDgERkZ3oByEAF9S6TLnyEU5CtTpNW7Lr371R4kBmx4CRMsQUMWbRJw4BxegFsbQzRN3SicozRPLzXa3Y0iFGCgAc3giUAAzACcIGyQyEBwIJDlFUxUOWHVHEGjYpAAmGiSkAGYPJTMLYNt7HLy4xHLE5MRUzyqubGC6sUiYxqKW+Ir07w5u+2EaRnoAIz4ABUk9GRAorGCACxxRShEgA
\begin{tikzcd}
\Gamma' \arrow[r, "\gamma"] \arrow[d, "\sigma'"'] & \Gamma \arrow[d, "\sigma"] \\
\Delta' \arrow[r, "\delta"]                       & \Delta                    
\end{tikzcd}
\]
where $\sigma$ and $\sigma'$ are the canonical isomorphisms between contexts. As a special case of the latter definition, we say that two dependent terms are canonically related if they are so as telescopic substitutions; namely, \textit{two terms $a \in A \; [\Gamma]$ and $b \in B \; [\Delta]$ are canonically related} if the dependent collections they belong to are canonically related and the following equality holds
\[
\tau(a) = b[\sigma] \in B[\sigma] \; [\Gamma]
\]
where $\sigma \in \Delta \; [\Gamma]$ is the canonical isomorphism between contexts, and $\tau(x) \in B[\sigma] \; [\Gamma, x \in A]$ is the canonical isomorphism between collections.
\end{definition}

\begin{remark}
We could have organised the definitions of canonical isomorphisms using the language of category theory. In particular, we could have considered the syntactic category $\mathbf{Ctx}$ of contexts and telescopic substitutions up to judgemental equality. In that case, the square depicted above in the definition of canonically related contexts could have been interpreted as a diagram of $\mathbf{Ctx}$, and formally required to be commutative.
\end{remark}

The first three points of Proposition \ref{can} imply that being canonically related (for contexts, collections, telescopic substitutions and terms) is an equivalence relation. Moreover, the following property of preservation under substitution holds.
\begin{lemma}\label{cansubpro}
Let $\gamma \in \Gamma \; [\Gamma']$ and $\delta \in \Delta \; [\Delta']$ be two canonically related telescopic substitutions. If $A \; type \;  [\Gamma]$ and $B \; type \;  [\Delta]$ are canonically related types, then also $A[ \gamma ] \; type \;  [\Gamma']$ and $B[ \delta ] \; type \;  [\Delta']$ are canonically related types. Moreover, if $a \in A \;  [\Gamma]$ and $b \in B \; [\Delta]$ are canonically related terms, then also $a[ \gamma ] \in A[ \gamma ] \;  [\Gamma']$ and $b[ \delta ] \in B[ \delta ] \;  [\Delta']$ are canonically related terms.
\end{lemma}
\begin{proof}
It follows from Lemma \ref{cansub} and Definition \ref{related}.
\end{proof}

Finally, we notice that we can always correct a type (resp. a term) into a canonically related one to match a given context (type) canonically equivalent to the original one.
\begin{lemma}\label{correct}
Let $A \; type \; [\Gamma]$, and $\Delta \; ctx$ canonically related to $\Gamma \; ctx$, then there exists $\Tilde{A} \; type \; [\Delta]$ canonically related to $A \; type \; [\Gamma]$. 

Analogously, if $a \in A \; [\Gamma]$ is a term and $B \; col \; [\Delta]$ is a collection canonically related to $A \; col \; [\Gamma]$, then there there exists a term $\Tilde{a} \in B \; [\Delta]$ canonically related to $a \in A \; [\Gamma]$.
\end{lemma}
\begin{proof}
Consider $\Tilde{A} :\equiv A[\sigma^{-1}] \; type \; [\Delta]$ and $\Tilde{a} :\equiv \tau(a)[\sigma^{-1}] \in B \; [\Delta]$, where $\sigma \in \Delta \; [\Gamma]$ and $\tau \in B[\sigma] \; [\Gamma, x \in A]$ are some  assumed existing canonical isomorphisms. The same $\sigma$ and $\tau$ witness that $A \; type \; [\Gamma]$ and $\Tilde{A} \; type \; [\Delta]$ are canonically related types, and that $a \in A \; [\Gamma]$ and $\Tilde{a} \in B \; [\Delta]$ are canonically related terms
\end{proof}

\subsection{Interpreting \texorpdfstring{$\mathbf{emTT}+\mathsf{propext}$}{emTT+propext} into \emtt}
With the machinery of canonical isomorphisms set up, we are ready to interpret $\mathbf{emTT}+\mathsf{propext}$ into \emtt. The idea is to define an identity interpretation \textit{up to canonical isomorphisms}.

As customary in type theory, we first define \textit{a priori partial} interpretation functions on the pre-syntax of $\mathbf{emTT}+\textsf{propext}$; the Validity Theorem \ref{validitypropext} will ensure that such functions are total when restricted to the derivable judgements of $\mathbf{emTT}+\textsf{propext}$. More in detail, we define three partial functions which send:
\begin{enumerate}
\item context judgements $\decorate{\Gamma} \; ctx$ to an equivalence class $\sem{\decorate{\Gamma} \; ctx}$ of canonically related $\mathbf{emTT}$-contexts;
\item type judgements $\decorate{A} \; type \; [\decorate{\Gamma}]$ to an equivalence class $\sem{\decorate{ A } \; type \; [\decorate{\Gamma }]}$ of canonically related $\mathbf{emTT}$-collections such that all its representatives are defined in contexts belonging to $\sem{\decorate{ \Gamma } \; ctx}$, and such that at least one among them is of kind $type$;
\item term judgements $\decorate{a} \in \decorate{A} \; [\decorate{\Gamma}]$ to an equivalence class $\sem{\decorate{a} \in \decorate{A} \; [\decorate{\Gamma}]}$ of canonically related $\mathbf{emTT}$-terms such that all its representatives are defined in contexts belonging to $\sem{\decorate{ \Gamma } \; ctx}$.
\end{enumerate}

In the following we use the parenthesis $\eqclass{-}$ to denote equivalence classes.

\begin{definition}[Interpretation]\label{intpre}
The three functions specified above are defined by recursion on pre-syntax of $\mathbf{emTT}+\textsf{propext}$ where, in each clause, we interpret the constructor in case (be it of contexts, types or terms) with the same constructor in the target theory \emtt. We spell out the case of contexts, variables, the canonical true term, and dependent product.
\\

\textit{Contexts and variables.}
\begin{itemize}
\item $\sem{() \; ctx} :\equiv \eqclass{() \; ctx}$
\item $\sem{\decorate{\Gamma}, x \in \decorate{A}  \; ctx} :\equiv \eqclass{ \Gamma' , x \in A' \; ctx}$

provided that $\sem{\decorate{A} \; col \; [\decorate{\Gamma}]} \equiv \eqclass{A' \; col \; [\Gamma']}$
\item $\sem{x \in A \; [ \decorate{\Gamma} , x \in \decorate{A} , \decorate{\Delta} ] } :\equiv \eqclass{x \in A' \; [\Gamma' , x \in A' , \Delta']}$

provided that $\sem{\decorate{\Gamma} , x \in \decorate{A} , \decorate{\Delta} \; ctx} \equiv \eqclass{\Gamma' , x \in A' , \Delta'  \; ctx}$
\end{itemize}

\textit{True term.}
\begin{itemize}
\item $\sem{\mathsf{true} \in \varphi \; [\Gamma]} :\equiv \eqclass{\mathsf{true} \in \varphi' \; [\Gamma']}$

provided that $\sem{\varphi \; prop \; [\Gamma]} :\equiv \eqclass{\varphi' \; prop \; [\Gamma']}$
\end{itemize}

\iffalse
\textit{Existential quantifier}
\begin{itemize}
\item $\sem{(\exists x \in \decorate{A})\decorate{\varphi} \,|\, \decorate{\Gamma}} :\equiv \eqclass{(\exists x \in A)\varphi \; prop \; [\Gamma]}$

provided that $\sem{\decorate{\varphi} \,|\, \decorate{\Gamma}, x \in \decorate{A}} \equiv \eqclass{\varphi \; prop \; [\Gamma, x \in A]}$
\end{itemize}
\fi

\textit{Dependent products.}
\begin{itemize}
\item
$\sem{(\Pi x \in \decorate{A})\decorate{B} \; type \; [\decorate{\Gamma}]} :\equiv \eqclass{(\Pi x \in A')B' \; type \; [\Gamma']}$
    
provided that $\sem{\decorate{B} \; type \; [\decorate{\Gamma}, x \in \decorate{A}]} \equiv \eqclass{B' \; type \; [\Gamma', x \in A']}$

\item $\sem{(\lambda x \in \decorate{A})\decorate{b} \in (\Pi x \in \decorate{A})\decorate{B} \; [\decorate{\Gamma}]} :\equiv \eqclass{(\lambda x \in A')b' \in (\Pi x \in A')B' \; [\Gamma']}$

provided that $\sem{\decorate{b} \in \decorate{B} \; [\decorate{\Gamma}, x \in \decorate{A}]} \equiv \eqclass{b' \in B'  \; [\Gamma', x \in A']}$

\item $\sem{\mathsf{Ap}(\decorate{f},\decorate{a}) \in B[a/x] \; [\decorate{\Gamma}]} :\equiv \eqclass{\mathsf{Ap}(f',a') \in B'[a'/x] \; [\Gamma']}$

provided that $\sem{\decorate{f} \in (\Pi x \in A)B \; [\decorate{\Gamma}]} \equiv \eqclass{f' \in (\Pi x \in A')B' \; [\Gamma']}$

and $\sem{\decorate{a} \in A \; [ \decorate{\Gamma}]} \equiv \eqclass{a' \in A' \; [\Gamma']}$

\end{itemize}
The interpretation of the other constructors is defined analogously.
\end{definition}

To smoothly state the substitution lemma, we define in an analogous way a fourth partial function sending judgements of the derived form $\decorate{\gamma} \in \decorate{\Gamma} \; [\decorate{\Delta}]$ to an equivalence class of \emtt-canonically related telescopic substitutions $\sem{\decorate{\gamma} \in \decorate{\Gamma} \; [\decorate{\Delta}]}$ defined in contexts belonging to $\sem{\decorate{\Delta} \; ctx}$. We then have the following.
\iffalse
\begin{itemize}
\item $\sem{() \,|\, \decorate{\Delta}} :\equiv \eqclass{() \in () \; [\Delta]}$

provided that $\sem{\decorate{\Delta}} \equiv \eqclass{\Delta \; ctx}$

\item
$\sem{\decorate{\gamma} , \decorate{a} \,|\, \decorate{\Delta}} :\equiv \eqclass{\gamma, a \in (\Gamma , x \in A) \; [\Delta]}$

provided that $\sem{\decorate{\gamma} \,|\, \decorate{\Delta}} \equiv \eqclass{\gamma \in \Gamma \; [\Delta]}$

and $\sem{\decorate{A} \,|\, \decorate{\Gamma}} \equiv \eqclass{A \; col \; [\Gamma]}$

and
$\sem{\decorate{a} \,|\, \decorate{\Delta}} \equiv \eqclass{a \in A[\gamma] \; [\Delta]}$
\end{itemize}
\fi

\begin{lemma}[Substitution]
\label{sub}
Assume $\sem{\decorate{\gamma} \in \decorate{\Gamma} \; [\decorate{\Delta}]} \equiv \eqclass{\gamma' \in \Gamma' \; [\Delta']}$ holds, then:
\begin{enumerate}
\item 
$\sem{\decorate{A} \; type \;[\decorate{\Gamma}]} \equiv \eqclass{A' \; type \; [\Gamma']}$ implies $\sem{\decorate{A}[\decorate{\gamma}] \; type \; [\decorate{\Delta}]} \equiv \eqclass{A'[\gamma'] \; type \; [\Delta']}$;
\item $\sem{\decorate{a} \in \decorate{A} \; [\decorate{\Gamma}]} \equiv \eqclass{a' \in A' \; [\Gamma']}$ implies

$\sem{\decorate{a}[\decorate{\gamma}] \in \decorate{A}[\decorate{\gamma}] \; [\decorate{\Delta}]} \equiv \eqclass{a'[\gamma'] \in A'[\gamma'] \; [\Delta']}$.
\end{enumerate}
\end{lemma}
\begin{proof}
By induction on the expressions $\decorate{A}$ and $\decorate{a}$.
\end{proof}

\begin{theorem}[Validity]\label{validitypropext}
\begin{enumerate}
\item if $\mathbf{emTT}+\mathsf{propext} \vdash \decorate{\Gamma} \; ctx$, then $\sem{\decorate{\Gamma}}$ is defined;
\item if $\mathbf{emTT}+\mathsf{propext} \vdash \decorate{A} \; type \; [\decorate{\Gamma}]$, then $\sem{\decorate{A} \; type \; [\decorate{\Gamma}]}$ is defined;
\item if $\mathbf{emTT}+\mathsf{propext} \vdash \decorate{a} \in \decorate{A} \; [\decorate{\Gamma}]$, then $\sem{\decorate{a} \in \decorate{A} \; [\decorate{\Gamma}]}$ is defined and all its terms are defined in types belonging to $\sem{\decorate{A} \; col \; [\decorate{\Gamma}]}$;
\item if $\mathbf{emTT}+\mathsf{propext} \vdash \decorate{A} = \decorate{B} \; type \; [\decorate{\Gamma}]$, then $\sem{\decorate{A} \; type \; [\decorate{\Gamma}]} \equiv \sem{\decorate{B} \; type \; [\decorate{\Gamma}]}$;
\item if $\mathbf{emTT}+\mathsf{propext} \vdash \decorate{a} = \decorate{b} \in \decorate{A} \; [\decorate{\Gamma}]$, then $\sem{\decorate{a} \in \decorate{A} \; [\decorate{\Gamma}]} \equiv \sem{\decorate{b} \in \decorate{A} \; [\decorate{\Gamma}]}$.
\end{enumerate}
\end{theorem}

\begin{proof}
By induction on the derivations of $\mathbf{emTT}+\mathsf{propext}$, using Proposition \ref{can} and Lemmas \ref{cansubpro}, \ref{correct}, and \ref{sub}. We spell out the relevant cases of lambda abstraction and propositional extensionality.

For the case of lambda abstraction, it is trivial to check that the \emtt-judgements used to interpret it are actually derivable, and that the side condition on the contexts holds. We also then need to check that the definition of the equivalence class does not depend on the choice of representatives. For that, assume that $b \in B \; [\Gamma, x \in A]$ and $b' \in B' \; [\Gamma', x \in A']$ are two canonically related terms. Thus, we know there are canonical isomorphisms
\begin{align*}
\sigma & \in \Gamma' \; [\Gamma] \\
\tau(x) & \in A'[\sigma ] \; [\Gamma, x \in A] \\
\rho(x,y) & \in B'[\sigma,\tau] \; [\Gamma,x \in A, y \in B]
\end{align*}
such that
\begin{equation}
\rho(x,b) = b'[\sigma, \tau] \in B[\sigma, \tau] \; [\Gamma, x \in A]
\end{equation}
By Proposition \ref{can} and Lemma \ref{cansub} also the following are canonical isomorphisms.
\begin{align*}
\tau^{-1}(x) & \in A \; [\Gamma , x \in A'[\sigma ]] \\
\rho(\tau^{-1}(x),y) & \in B'[\sigma,x] \; [\Gamma,x \in A'[\sigma], y \in B(\tau^{-1}(x))]
\end{align*}
Moreover, by applying the term $\rho^{-1}$ to (1) also the followings hold.
\begin{align*}
b = \rho^{-1}(x,b'[\sigma, \tau]) & \in B \; [\Gamma, x \in A] \\
(\lambda x \in A)b = (\lambda x \in A)\rho^{-1}(x,b'[\sigma, \tau]) & \in (\Pi x \in A)B \; [\Gamma]
\end{align*}

By definition of canonical isomorphism between dependent products, we have that the term
\[
\zeta(f) :\equiv (\lambda x \in A'[\sigma])s(\tau^{-1}(x),\mathsf{Ap}(f,\tau^{-1}(x)))
\]

is a canonical isomorphisms between $(\Pi x \in A)B$ and $(\Pi x \in A'[\sigma])B'[\sigma,x] \equiv ((\Pi x \in A')B')[\sigma]$. Finally, we can check that
\begin{align*}
\zeta((\lambda x \in A)b) & = \zeta((\lambda x \in A)\rho^{-1}(x,b'[\sigma, \tau])) \\
& = (\lambda x \in A'[\sigma])\rho(\tau^{-1}(x), \rho^{-1}(\tau^{-1}(x),b'[\sigma ,\tau][\tau^{-1}/x])) \\
& = (\lambda x \in A'[\sigma])b'[\sigma, x]\\
& \equiv ((\lambda x \in A')b')[\sigma] \in ((\Pi x \in A')B')[\sigma] \; [\Gamma]
\end{align*}
Thus, we have concluded that $(\lambda x \in A)b \in (\Pi x \in A)B \; [\Gamma]$ and $(\lambda x \in A')b' \in (\Pi x \in A')B' \; [\Gamma']$ are canonically related terms.
\\

Propositional extensionality is validated as follows. By inductive hypothesis on the first premise, we know that, for some $\varphi' \; prop \; [\Gamma']$, we have $\sem{\decorate{\varphi} \; prop \; [\decorate{\Gamma}]}\equiv \eqclass{\varphi' \; prop \; [\Gamma']}$; by inductive hypothesis on the second premise corrected by Lemma \ref{correct}, we know that $\sem{\decorate{\psi} \; prop \; [\decorate{\Gamma}]}\equiv \eqclass{\psi' \; prop \; [\Gamma']}$ for some proposition $\psi'$ defined in the same context $\Gamma'$ of $\varphi'$. By definition of the interpretation we then have
$
\sem{\decorate{\varphi} \Leftrightarrow \decorate{\psi} \; prop \; [\decorate{\Gamma}]} \equiv \eqclass{\varphi' \Leftrightarrow \psi' \; prop \; [\Gamma']}$.
Finally, by inductive hypothesis on the third premise, we know that the interpretation of $\sem{\mathsf{true} \in \decorate{\varphi} \Leftrightarrow \decorate{\psi} \; prop \; [\decorate{\Gamma}]}$ is well defined; in particular, this means that $\mathsf{true} \in \varphi' \Leftrightarrow \psi'$ is derivable in \emtt, but this amounts to $\varphi'$ and $\psi'$ being canonically related, which in turn implies
$
\sem{\decorate{\varphi} \; prop \; [\decorate{\Gamma}]}\equiv \eqclass{\varphi' \; prop \; [\Gamma']} \equiv \eqclass{\psi' \; prop \; [\Gamma']} \equiv \sem{\decorate{\psi} \; prop \; [\decorate{\Gamma}]}
$.
\end{proof}

The interpretation enjoys the following property, which allows it to be seen as a retraction of the identity interpretation of $\mathbf{emTT}$ into $\mathbf{emTT}+\mathsf{propext}$.

\begin{proposition}\label{cons}
\begin{enumerate}
\item If $\mathbf{emTT} \vdash \Gamma \; ctx$, then $\sem{\Gamma \; ctx} \equiv \eqclass{\Gamma \; ctx}$;
\item if $\mathbf{emTT} \vdash A \; type \; [\Gamma]$, then $\sem{A \; type \; [\Gamma]} \equiv \eqclass{A \; type \; [\Gamma]}$;
\item if $\mathbf{emTT} \vdash a \in A \; [\Gamma]$, then $\sem{a \in A \; [\Gamma]} \equiv \eqclass{a \in A \; [\Gamma]}$.
\end{enumerate}
\end{proposition}
\begin{proof}
Straightforward by induction on the derivations of $\mathbf{emTT}$.
\end{proof}

\begin{corollary}[Conservativity]\label{consprop}
If $\mathbf{emTT} \vdash \varphi \; prop \; [\Gamma]$ and $\mathbf{emTT}+\mathsf{propext} \vdash \varphi \; \mathsf{true}\; [\Gamma]$, then $\mathbf{emTT} \vdash \varphi \; \mathsf{true}\; [\Gamma]$.
\end{corollary}
\begin{proof}
By point 2 of Proposition \ref{cons}, $\sem{\varphi \; prop \; [\Gamma]} \equiv \eqclass{\varphi \; prop \; [\Gamma]}$; then, we conclude by point 3 of Theorem \ref{validitypropext}.
\end{proof}

\begin{corollary}[Equiconsistency]\label{equi}
The theories $\mathbf{mTT}$ and $\mathbf{emTT}$ are equiconsistent.
\end{corollary}
\begin{proof}
From \cite{m09}, we know that the consistency of \mtt\ implies that of \emtt. For the other direction, 
we know that \mtt\ can be interpreted in  $\mathbf{emTT}+\mathsf{propext}$ by 
 Proposition \ref{eqc1}; in turn, $\mathbf{emTT}+\mathsf{propext}$ is equiconsistent to its fragment \emtt\ by Corollary \ref{consprop}.
\end{proof}

\section{Equiconsistency of \mf\ with its classical version}\label{eqclass}
In this section, we adapt the Gödel-Gentzen's double-negation translation of classical predicative logic into the intuitionistic one in \cite{Troelstra} to interpret the classical version \emttc\ of the  extensional level \emtt\ of \mf\ into \emtt\ itself. More in details, 
\begin{definition}
\emttc\ is  the extension of \emtt\ with the following rule formalising the Law of Excluded Middle.
\[
\mathsf{LEM}\;\frac
{\varphi \; prop}
{\varphi \vee \neg \varphi \; \mathsf{true}}
\]
\end{definition}

The underlying idea of the translation is straightforward: we want to keep translating propositions of \emttc\ into stable propositions  $\varphi$ of \emtt\  (i.e. those equivalent to their the double negation)  while leaving unaltered set-theoretical constructors that do not involve logic. Formally, a proposition $\varphi \; prop$ of \emtt\  is said to be \textit{stable} if the judgement $\neg\neg \varphi \Rightarrow \varphi \; true$ is derivable. Accordingly, a collection $A \; col$ is said to have \textit{stable equality} if its propositional equality $\mathsf{Eq}(A,x,y) \; prop \; [x,y \in A]$ is stable in \emtt.

Since \emtt\  is a dependent type system in which sorts can depend on terms and propositions, the translation will be defined on all those entities, and not just on formulas.

\begin{definition}[Translation of \emttc\ into \emtt]
\label{thedntranslation}
We define by simultaneous recursion four endofunctions $(-)^{\mathcal{N}}$ on pre-contexts, pre-types, pre-propositions, and pre-terms, respectively.

\textit{Variables and contexts.} The translation does not affect variables, and on contexts it is defined in the obvious way.
\[
x^{\mathcal{N}} :\equiv x \qquad
()^{\mathcal{N}} :\equiv () \qquad
(\Gamma, x \in A)^{\mathcal{N}} :\equiv \Gamma^{\mathcal{N}},\, x \in A^{\mathcal{N}}
\]

\textit{Logic.} We translate the connectives as in the case of predicate logic, but in the case of quantifiers the translation is recursively applied also to the domain of quantification. Contrary to the case of predicate logic, we do not double-negate the propositional equality, and we recursively applied the translation also to its type and terms; in the validity theorem, the burden of proving that it is stable is transferred to the translation of types. Finally, the $\mathsf{true}$ term is mapped to itself.
\begin{equation*}
\begin{aligned}[c]
\bot^{\mathcal{N}} & :\equiv \bot \\
(\varphi \wedge \psi)^{\mathcal{N}} & :\equiv \varphi^{\mathcal{N}} \wedge \psi^{\mathcal{N}} \\
(\varphi \Rightarrow \psi)^{\mathcal{N}} & :\equiv \varphi^{\mathcal{N}} \Rightarrow \psi^{\mathcal{N}} \\
(\varphi \vee \psi)^{\mathcal{N}} & :\equiv \neg\neg(\varphi^{\mathcal{N}} \vee \psi^{\mathcal{N}})
\end{aligned}
\qquad\qquad
\begin{aligned}[c]
((\exists x \in A)\varphi)^{\mathcal{N}} & :\equiv \neg\neg(\exists x \in A^{\mathcal{N}})\varphi^{\mathcal{N}} \\
((\forall x \in A)\varphi)^{\mathcal{N}} & :\equiv (\forall x \in A^{\mathcal{N}})\varphi^{\mathcal{N}} \\
\mathsf{Eq}(A,a,b)^{\mathcal{N}} & :\equiv
\mathsf{Eq}(A^\mathcal{N},a^{\mathcal{N}},b^{\mathcal{N}}) \\
\mathsf{true}^{\mathcal{N}} & :\equiv \mathsf{true}
\end{aligned}
\end{equation*}

\textit{Set constructors.} Since we do not want to alter set-theoretic constructions, we just recursively apply the translation to their sub-expressions. We report here the cases of the empty set and dependent sums; the same (trivial) pattern will apply also to the pre-syntax of singleton set, disjoint sums, dependent products, lists and quotients.
\begin{equation*}
\begin{aligned}[c]
\mathsf{N_0}^{\mathcal{N}} & :\equiv \mathsf{N_0}\\
\mathsf{El}_{\mathsf{N_0}}(c)^{\mathcal{N}} & :\equiv \mathsf{El}_{\mathsf{N_0}}(c^{\mathcal{N}}) \\
&
\end{aligned}
\qquad\qquad
\begin{aligned}[c]
((\Sigma x \in A)B)^{\mathcal{N}} & :\equiv (\Sigma x \in A^{\mathcal{N}})B^{\mathcal{N}} \\
\langle a, b \rangle^{\mathcal{N}} & :\equiv \langle a^{\mathcal{N}}, b^{\mathcal{N}} \rangle \\
\mathsf{El}_{\Sigma}(d,(x,y).m)^{\mathcal{N}} & :\equiv \mathsf{El}_{\Sigma}(d^{\mathcal{N}},(x,y).m^{\mathcal{N}})
\end{aligned}
\end{equation*}

\textit{Power collection of the singleton.} We translate the power collection of the singleton into its subcollection of stable propositions (up to equiprovability); the translation of its introduction constructor just accounts for this change.
\[
\mathcal{P}(1)^{\mathcal{N}} :\equiv (\Sigma\, x \in \mathcal{P}(1))(\neg\neg\ext(x) \Rightarrow \ext(x)) \qquad [\varphi]^{\mathcal{N}} :\equiv \langle [\varphi^{\mathcal{N}}] , \mathsf{true} \rangle
\]
\end{definition}
This concludes the definition of the translation. We immediately notice that it enjoys the following syntactical property, which we will tacitly exploit in the rest of the discussion.
\begin{lemma}[Substitution]
If $e$ and $t$ are two expressions of the pre-syntax, and $x$ is a variable, then $(e[t/x])^{\mathcal{N}} \equiv e^{\mathcal{N}}[t^{\mathcal{N}}/x]$.
\end{lemma}
\begin{proof}
Straightforward, by induction on the pre-syntax.
\end{proof}

The next two propositions will be vital to prove the validity theorem. They characterise the equality of various type constructors of $\mathbf{emTT}$ and collect a series of closure properties for collections with stable equality, respectively.

\begin{proposition}\label{char}
The following equivalences hold in $\mathbf{emTT}$ (where the free variables in the left-hand side of each equivalence are implicitly assumed to be in the obvious context): 
\begin{enumerate}

\item $x =_{\mathsf{N_0}} y \Leftrightarrow \bot$

\item $x =_{\mathsf{N_1}} y \Leftrightarrow \top$

\item $l =_{\mathsf{List}(A)} l' \Leftrightarrow \begin{cases}
\top & \text{if $l = l' = \epsilon$} \\
s =_{\mathsf{List}(A)} s' \wedge a =_A a' & \text{if $l = \mathsf{cons}(s,a)$ and $l' = \mathsf{cons}(s',a')$} \\ 
\bot & \text{otherwise}
\end{cases}$

where the proposition defined by cases on the right side can be formally defined using the elimination of lists applied toward the collection ${\cal P}(1)$.\footnote{Namely as $
\ext(\mathsf{El_{List}}(l, \mathsf{El_{List}}(l', [\top], [\bot]),(x,y,z).c))$, where $l,l' \in \mathsf{List}(A)$ and $c(x,y,z) :\equiv \mathsf{El_{List}}(l',[\bot],(x',y',z').[\ext(z) \wedge y =_A y'])$.}

%begin{align*}
%c(x,y,z) & \in \mathcal{P}(1) \; [x \in \mathsf{List}(A) , y \in A , z \in \mathcal{P}(1)] \\
%c(x,y,z) & :\equiv \mathsf{El_{List}}(l',[\bot],(x',y',z').[\ext(z) \wedge y =_A y'])
%\end{align*}
%\[\ext(\mathsf{El_{List}}(l, \mathsf{El_{List}}(l', [\top], [\bot]),(x,y,z).\mathsf{El_{List}}(l',[\bot],(x',y',z').[\ext(z) \wedge y =_A y'])))\]

\item 
$z =_{A+B} z' \Leftrightarrow \begin{cases}
x =_A x' & \text{if $z = \mathsf{inl}(x)$ and $z' = \mathsf{inl}(x')$} \\
y =_B y' & \text{if $z = \mathsf{inr}(y)$ and $z' = \mathsf{inr}(y')$} \\ 
\bot & \text{otherwise}
\end{cases}$

where the proposition defined by cases on the right side is formally defined analogously as in the previous point;
% \[\ext(\mathsf{El_{+}}(z,(x).\mathsf{El_{+}}(z,(x').[x=_A x'],(y').[\bot]),(y).\mathsf{El_{+}}(z,(x').[\bot],(y').[y =_B y'])))\]

\item $\langle a,b \rangle =_{(\Sigma x \in A)B(x)} \langle a',b' \rangle \Leftrightarrow (\exists x \in a =_A a')\ b =_{B(a)} b'$

\item $f =_{(\Pi x \in A)B(x)} g \Leftrightarrow (\forall x \in A)\ f(x)=_{B(x)}g(x)$

\item $[a] =_{A/R} [b] \Leftrightarrow R(a,b)$

\item  $U =_{\mathcal{P}(1)} V \Leftrightarrow (\ext(U) \Leftrightarrow \ext(V))$

\item  $p =_{\varphi} q \Leftrightarrow \top \quad$ if $\varphi \; prop$
\end{enumerate}
\end{proposition}

\begin{proof}
Using the judgemental equality rules and the (possibly derived) $\eta$-rules of the corresponding constructors. Additionally, in the cases of lists and disjoint sums one uses induction principles together with the standard trick of eliminating toward the collection ${\cal P}(1)$ to prove the disjointedness of their term constructors.
\end{proof}
\begin{remark}
Since we are in an extensional type theory, the above proposition works smoothly especially for the clause of the dependent sum. Notice in fact that the proposition on the right-hand side of point 5 could not have been written simply as the conjunction $a =_A a' \,\wedge\, b =_{B(a)} b'$, which is ill-formed since the judgement $b' \in B(a)$ cannot be derived without having proved $a =_A a'$ first.

%In the intensional level a dependent setoid $(B(x),=_{B(x)},\sigma_B)$ over the setoid $(A,=_A)$, the proposition would have been rendered as
%\[ (\exists p \in a =_A a')\sigma(a,a',p,b) =_{B(a')} b\]
\end{remark}

\begin{proposition}\label{two}

In $\mathbf{emTT}$, propositions, the empty set, the singleton set, and the collection $(\Sigma\, x \in \mathcal{P}(1))(\neg\neg\ext(x) \Rightarrow \ext(x))$ have stable equality; moreover, having stable equality is preserved by taking lists, disjoint sums, dependent sums, and dependent products; finally, a set quotiented by a stable equivalence relation has stable equality.
\end{proposition}
\begin{proof}
All cases are proved using Proposition \ref{char}. We spell out only the most interesting ones.

For the case of dependent sum, assume $A \;col$ and $B(x) \; col \; [x \in A]$ to be two collections with stable equality. By Proposition \ref{char}, we must prove that for terms $a,a' \in A$, $b \in B(a)$, and $b' \in B(a')$ the proposition $(\exists x \in a =_A a') \ b =_{B(a)} b' $ is stable. By the elimination rule of the existential quantifier, we can derive the followings.
\begin{align*}
& (\exists x \in a =_A a')\;b =_{B(a)} b' \quad \Rightarrow \quad a =_A a' \\
& (\forall y \in a =_A a')(\ (\exists x \in a =_A a') \; b =_{B(a)} b' \quad \Rightarrow \quad b =_{B(a)} b' \;)
\end{align*}

From these, we get
\begin{align}
& \neg\neg(\exists x \in a =_A a')\; b =_{B(a)} b' \quad \Rightarrow \quad  \neg\neg a =_A a'\\
& (\forall y \in a =_A a')(\;\neg\neg (\exists x \in a =_A a')\; b =_{B(a)} b' \quad \Rightarrow \quad\neg\neg  b =_{B(a)} b'\;)
\end{align}
Assume $\neg\neg (\exists x \in a =_A a')\; b =_{B(a)} b'$; from (2) we deduce $\neg\neg a =_A a'$
and, since $A$ has stable equality, we conclude $a =_A a'$; knowing that $a =_A a'$ holds, we can now apply the hypothesis to (3) and we deduce
$\neg\neg b =_{B(a)} b'$, which, since $B(a)$ has stable equality for any given $a$, implies $b =_{B(a)} b'$; finally, by the introduction rule of the existential quantifier we have $(\exists x \in a =_A a')\; b =_{B(a)} b'$. Hence, we have shown that the proposition $(\exists x \in a =_A a')\; b =_{B(a)} b'$ is stable.
\\

For the collection $(\Sigma x \in \mathcal{P}(1))(\neg\neg\ext(x)\Rightarrow \ext(x))$ we have that, by the rules for equality of dependent pairs and propositions in Proposition \ref{char}, its propositional equality is equivalent to \[ \pi_{\mathsf{1}}(z) =_{\mathcal{P}(1)} \pi_{\mathsf{1}}(w) \quad \text{with } z,w \in (\Sigma x \in \mathcal{P}(1))(\neg\neg\ext(x)\Rightarrow \ext(x))\] which, again by the case of ${\cal P}(1)$ in Proposition \ref{char}, is in turn equivalent to \[ \ext(\pi_1(z)) \Leftrightarrow \ext(\pi_1(w))\]
Since the propositions $\ext(\pi_1(z))$ and $\ext(\pi_1(w))$ are stable (by $\pi_2(z)$ and $\pi_2(w)$, respectively), and since conjunction and implication preserve stability, we conclude that $(\Sigma x \in \mathcal{P}(1))(\neg\neg\ext(x)\Rightarrow \ext(x))$ has stable equality.
\end{proof}
    
We are now ready to prove the validity of the interpretation.

\begin{theorem}[Validity]\label{dnvalidity}
The translation is an interpretation of \emttc\ into \emtt, in the sense that it preserves judgement derivability between the two theories:
\begin{enumerate}
\item
if \emttc\ $\vdash \Gamma \; ctx$, then $\mathbf{emTT} \vdash \Gamma^{\mathcal{N}} \; ctx$

\item
if \emttc\ $\vdash A \; type \; [\Gamma]$, then $\mathbf{emTT} \vdash A^{\mathcal{N}} \; type \; [\Gamma^{\mathcal{N}}]$

\item
if \emttc\ $\vdash a \in A \; [\Gamma]$, then $\mathbf{emTT} \vdash a^{\mathcal{N}} \in A^{\mathcal{N}} \; [\Gamma^{\mathcal{N}}]$

\item
if \emttc\ $\vdash A = B \; type \; [\Gamma]$, then $\mathbf{emTT} \vdash A^{\mathcal{N}} = B^{\mathcal{N}} \; type \; [\Gamma^{\mathcal{N}}]$

\item
if \emttc\ $\vdash a = b \in A \; [\Gamma]$, then $\mathbf{emTT} \vdash a^{\mathcal{N}} = b^{\mathcal{N}} \in A^{\mathcal{N}} \; [\Gamma^{\mathcal{N}}]$
\end{enumerate}
Moreover, the translation produces stable propositions and, in particular, collections with stable equality:
\begin{enumerate}\setcounter{enumi}{5}
\item
if \emttc\ $\vdash \varphi \; prop \; [\Gamma]$, then \[\mathbf{emTT} \vdash \neg\neg\varphi^{\mathcal{N}} \Rightarrow \varphi^{\mathcal{N}} \; \mathsf{true} \; [\Gamma^{\mathcal{N}}]\]
\item
if \emttc\ $\vdash A \; col \; [\Gamma]$, then \[\mathbf{emTT} \vdash \neg\neg\mathsf{Eq}(A^{\mathcal{N}}, x, y) \Rightarrow  \mathsf{Eq}(A^{\mathcal{N}}, x, y) \; \mathsf{true} \; [\Gamma^{\mathcal{N}}, x \in A^{\mathcal{N}}, y \in A^{\mathcal{N}}]\]
\end{enumerate}
\end{theorem}
\begin{proof}
All seven points are proved simultaneously by induction on judgements derivation. Due to the trivial pattern of the translation on most of the constructors, the majority of cases are trivially checked. For point 7, it suffices to apply the inductive hypotheses using Proposition \ref{two}. The cases involving the axiom of the excluded middle, the falsum constant, the disjunction and the existential quantifier are checked as in the case of translating classical predicate logic into the intuitionistic one; in the case of propositional equality, point 6 is checked using the inductive hypothesis on point 7.
\end{proof}

\begin{corollary}\label{equiconsistency}
The theories \emttc\ and $\mathbf{emTT}$ are equiconsistent.
\end{corollary}
\begin{proof}
By point 3 of Theorem \ref{validitypropext}, since the inconsistency judgement $\mathsf{true} \in \bot \; [\,]$ is sent by the translation to itself.
\end{proof}

%\begin{remark}
%The same idea could be adapt to work also at the intensional level extended with the rules $\textsf{prop-mono}$ and $\textsf{prop-true}$, by translating each proof-term constructors into $\mathsf{true}$, the intensional equality to its double-negation (as in the case of predicate logic), the codes of props a la tarsky should internalize the translation. The only additional check is indeed that on intensional equality, but now the translation is not required to produce collection with stable equality.
%\end{remark}

\begin{remark}
Observe that the above proofs go well within the extensional type theory. Interpreting \emttc\ directly into the intensional level \mtt\ would have been more complicated whilst possible with the use of canonical isomorphisms.
\end{remark}

\begin{remark}
Among theories that exploit dependent types, the double-negation translation has been applied also to the Calculus of Constructions in \cite{CoCprooftheory}, and to logic-enriched type theories in \cite{luo}.
The first result will be extended in Section \ref{impmf} by considering an extension of the base calculus with inductive types.
In contrast, the second calculus is closer to a multi-sorted logic, where propositions are not types, and there is no comprehension of a type with a proposition.
Our result shows that the double-negation translation goes through when logic is part of the type theory, mainly because of comprehensions, quotients, and equality reflection.
\end{remark}

\iffalse
As a final observation, we notice that the interpretation of the previous section which uses canonical isomorphisms to show the conservativity of propositional extensionality scales without any additional difficulty to the classical case.

\begin{corollary}\label{validityc}
The theory $\mathbf{emTT}^c+\mathsf{propext}$ is conservative over $\mathbf{emTT}^c$.
\end{corollary}
\begin{proof}
It is easy to check that the rule $\mathsf{LEM}$ is validated by the interpretation of Definition \ref{intpre}, and that Theorem \ref{validitypropext} and Corollary \ref{consprop} still hold when the theories in the statement are replaced with their classical versions.
\end{proof}
\fi

\section{Compatibility of \mf\ with classical predicativism à la Weyl}
As an application of the equiconsistency between $\mathbf{emTT}$ and \emttc, in this subsection we deduce that, accordingly to Weyl's treatment of classical predicative mathematics \cite{DasKontinuum}, neither Dedekind real numbers nor number-theoretic functional relations form a set.

We start by observing that, although classical, in \emttc\ the type of booleans $\mathsf{Bool}:\equiv \mathsf{N_1}+\mathsf{N_1}$ is not a propositional classifier; this is because, even in the presence of the excluded middle, we cannot eliminate from a disjunction $\varphi \vee \neg \varphi$ towards the set $\mathsf{N_1}+\mathsf{N_1}$. More generally, we have the following result.

\begin{proposition}\label{pn1}
In \emttc, and thus also in $\mathbf{emTT}$, the power collection of the singleton $\mathcal{P}(1)$ is not isomorphic to any set.
\end{proposition}
\begin{proof}
If $\mathcal{P}(1)$ were isomorphic to a set, then each collection of \emttc\ would be, and in particular, $\mathcal{P}(\mathbb{N})$. Thus, we could interpret full second-order arithmetic in \emttc; but this is a contradiction since, by Corollary \ref{equiconsistency}, we know that the proof-theoretic strength of \emttc\ coincide with that of $\mathbf{emTT}$, which, in turn, is bounded by $\widehat{\mathbf{ID}}_1$ as shown in \cite{IMMS}.
\end{proof}

In $\mathbf{emTT}$, the collection of Dedekind real numbers can be defined from the collection of subsets of rational numbers $\mathcal{P}(\mathbb{Q})$ through Dedekind  (left) cuts as
\begin{align*}
\mathbb{R} :\equiv (\Sigma A \in \mathcal{P}(\mathbb{Q}))( & (\exists q \in \mathbb{Q})q\,\varepsilon\,A  \\
& \wedge (\exists q \in \mathbb{Q})\neg q\,\varepsilon\,A \\
& \wedge (\forall q \,\varepsilon\, A)(\forall r \in \mathbb{Q})(r < q \Rightarrow r\,\varepsilon\,A) \\
& \wedge (\forall q \,\varepsilon\, A)(\exists r \,\varepsilon\, A)q < r)
\end{align*}
analogously, the collection of number-theoretic functional relations can be constructed from $\mathcal{P}(\mathbb{N} \times \mathbb{N})$ as
\[
\mathsf{FunRel}(\mathbb{N},\mathbb{N}) :\equiv (\Sigma R \in \mathcal{P}(\mathbb{N}\times \mathbb{N}))(\forall x \in \mathbb{N})(\exists ! y \in \mathbb{N})R(\langle x,y \rangle)
\]
The following shows that both $\mathbb{R}$ and $\mathsf{FunRel}(\mathbb{N},\mathbb{N})$ are \textit{proper} collections.

\begin{theorem}\label{rnotaset}
In \emttc, and thus also in $\mathbf{emTT}$, the collections $\mathbb{R}$ and $\mathsf{FunRel}(\mathbb{N},\mathbb{N})$ are not isomorphic to any set.
\end{theorem}
\begin{proof}
If $\mathbb{R}$ were isomorphic to a set, then the set $\{0,1\}_\mathbb{R}$ obtained from $\mathbb{R}$ by comprehension through the small proposition
\[
(\forall q \in \mathbb
{Q})(q \,\varepsilon\, A \Leftrightarrow q < 0) \vee (\forall q \in \mathbb
{Q})(q \,\varepsilon\, A \Leftrightarrow q < 1) \text{ with $A \in \mathcal{P}(\mathbb{Q})$}
\]
would be isomorphic to a set too. In turn, it is easy to show that, classically, $\mathcal{P}(1)$ is isomorphic to $\{0,1\}_\mathbb{R}$; the isomorphism is obtained by specialising to $\{0,1\}_\mathbb{R}$ the following operations between $\mathcal{P}(1)$ and $\mathcal{P}(\mathbb{Q})$.
\begin{align*}
[(\forall q \in \mathbb
{Q})(q \,\varepsilon\, A \Leftrightarrow q < 1)] & \in \mathcal{P}(1) \; & [A \in \mathcal{P}(\mathbb{Q})] \\
\{q \in \mathbb{Q} \,|\, (q < 0 \wedge \neg \ext(x)) \vee (q < 1 \wedge \ext(x)) \} & \in \mathcal{P}(\mathbb{Q}) \; & [x \in \mathcal{P}(1)]
\end{align*}
We conclude by Proposition \ref{pn1}.

For $\mathsf{FunRel}(\mathbb{N},\mathbb{N})$ the proof is analogous, using the set obtained by comprehension from it through the small proposition
$R(\langle x,y \rangle) \Rightarrow y =_\mathbb{N} 0 \vee y =_\mathbb{N} 1$, with $R \in \mathcal{P}(\mathbb{N} \times \mathbb{N})$.
\end{proof}

\section{Equiconsistency of the Calculus of Constructions with its classical version}\label{impmf}

Recall that  the intensional level \mtt\ of \mf\ can be seen as a predicative version of the Calculus of Constructions in \cite{CoC}. More precisely, consider the impredicative theory $\mathbf{mTT}_\mathsf{imp}$ obtained by extending the intensional level \mtt\ with the congruence rules for types and terms and with the following resizing rules collapsing the predicative distinction between effective and open-ended types.
\[
\textsf{col-into-set}\;\frac{A \; col}{A \; set} \qquad \textsf{prop-into-prop}_\mathsf{s}\;\frac{\varphi \; prop}{\varphi \; prop_s}
\]
Analogously, consider the impredicative theory $\mathbf{emTT}_\mathsf{imp}$ obtained by extending \emtt\ with the same resizing rules above.

The theories $\mathbf{mTT}_\mathsf{imp}$ and $\mathbf{emTT}_\mathsf{imp}$ can be interpreted as an extended version of the Calculus of Constructions with inductive types from \mltt, and an extensional version of it with the quotient constructor, respectively.

In particular, thanks to the resizing rules, the universe of small proposition $\mathsf{Prop_s}$ of \mtt\ becomes an impredicative universe of (all) propositions. For example, we can derive impredicative quantification as shown in the following derivation tree (where, for readability, we use $\mathsf{Prop_s}$ presented à la Russel).
\begin{prooftree}
\AxiomC{$\varphi(x) \in \mathsf{Prop_s} \; [x \in \mathsf{Prop_s}]$}
\LeftLabel{$\textsf{E-}\mathsf{Prop_s}$}
\UnaryInfC{$\varphi(x) \; prop_s \; [x \in \mathsf{Prop_s}]$}
\AxiomC{$\mathsf{Prop_s} \; col$}
\LeftLabel{$\textsf{F-}\forall$}
\BinaryInfC{$(\forall x \in \mathsf{Prop_s})\varphi(x) \; prop$}
\LeftLabel{$\textsf{prop-into-prop}\mathsf{_s}$}
\UnaryInfC{$(\forall x \in \mathsf{Prop_s})\varphi(x) \; prop_s$}
\LeftLabel{$\textsf{I-}\mathsf{Prop_s}$}
\UnaryInfC{$(\forall x \in \mathsf{Prop_s})\varphi(x) \in \mathsf{Prop_s}$}
\end{prooftree}
\iffalse
\begin{prooftree}
\AxiomC{$\varphi(x) \in \mathsf{Prop_s} \; [x \in \mathsf{Prop_s}]$}
\AxiomC{$\mathsf{Prop_s} \; col$}
\LeftLabel{$\textsf{F-}\forall$}
\BinaryInfC{$(\forall x \in \mathsf{Prop_s})\varphi(x) \; prop$}
\LeftLabel{$\textsf{prop-into-prop}\mathsf{_s}$}
\UnaryInfC{$(\forall x \in \mathsf{Prop_s})\varphi(x) \in \mathsf{Prop_s}$}
\end{prooftree}
\fi

Formally, we denote with \cocp\ the Calculus of Constructions without universes of types (apart from the impredicative universe of propositions) defined in \cite{CoC}, extended with rules for the inductive type constructors $\mathsf{N_0}$, $\mathsf{N_1}$, $+$, $\mathsf{List}$, and $\Sigma$ from the first-order fragment of \mltt\ (notice that the resulting theory is a rather small fragment of the Calculus of Inductive Constructions \cite{CIC}).

\begin{proposition}\label{a1}
$\mathbf{mTT}_\mathsf{imp}$ coincides with \cocp.
\end{proposition}
\begin{proof}
Since in $\mathbf{mTT}_\mathsf{imp}$ the distinction between sets and collections, as well as propositions and small propositions, disappears we have that the universe of small proposition $\mathsf{Prop_s}$ becomes the impredicative universe of (all) propositions; set constructors are available for all types, as in \cocp; and the universal quantifier $\forall$ and the dependent function space $\Pi$ are just two names for the only $\Pi$ constructor of \cocp. Then, the only calculations to be made are those to check that the propositional constructors coincide with their impredicative encoding made from the universal quantifier alone; in particular, it works for identity since in \mtt\ is defined à la Leibniz.
\end{proof}

Remarkably, the addition of impredicativity to \mf\ does not affect most of the techniques used to investigate its meta-mathematical properties. In particular, the quotient model, the equiconsistency of the two levels, and the equiconsistency with the classical version all scale easily to the impredicative case.

\begin{proposition}\label{a2}
The theory \emttimp\ is interpretable in the quotient model constructed over \mttimp.
\end{proposition}
\begin{proof}
By using the same interpretation defined in \cite{m09}. The additional resizing rules of \emttimp\ are easily validated. For example, consider the rule $\textsf{col-into-set}$; to check its validity we need to know that, for each \emttimp-collection $A$, the dependent extensional collection $A_=^\mathcal{I}$ interpreting it is a dependent extensional set, which, by definition, amounts to know that its support $A^\mathcal{I}$ is a set and its equivalence relation $=_{A^\mathcal{I}}$ is a small proposition; but this is guaranteed precisely by the resizing rules of \mttimp.
\end{proof}

\begin{corollary}\label{a8}
The theory $\mathbf{emTT}_\mathsf{imp}$ is interpretable in the quotient model constructed over \cocp.
\end{corollary}
\begin{proof}
Combining Propositions \ref{a1} and \ref{a2}.
\end{proof}

\begin{proposition}\label{a5}
$\mathbf{emTT}_\mathsf{imp} + \mathsf{propext}$ is conservative over $\mathbf{emTT}_\mathsf{imp}$, and $\mathbf{emTT}^c_\mathsf{imp} + \mathsf{propext}$ is conservative over $\mathbf{emTT}^c_\mathsf{imp}$.
\end{proposition}
\begin{proof}
Since canonical isomorphisms has been defined inductively in the meta-theory, and not internally as in \cite{MFhott}, we can use the same interpretation described in Definition \ref{intpre}. In the second point of the Validity Theorem \ref{validitypropext}, the additional resizing rules of the source theories are validated thanks to the same rules in the corresponding target theory; in the third point of the same theorem, the additional axiom $\mathsf{LEM}$ is validated analogously, thanks to the fact that the interpretation fixes the connectives: $\sem{\decorate{\varphi} \vee \neg \decorate{\varphi}} \equiv \eqclass{\varphi \vee \neg \varphi}$ whenever $\sem{\decorate{\varphi}} \equiv \eqclass{\varphi}$. By the same observations, Proposition \ref{cons} still holds in the presence of resizing rules and of $\mathsf{LEM}$. Then, we can conclude as in Corollary \ref{consprop}.
\end{proof}

\begin{proposition}\label{a6}
The theories $\mathbf{emTT}^c_\mathsf{imp}$ and $\mathbf{emTT}_\mathsf{imp}$ are equiconsistent.
\end{proposition}
\begin{proof}
By using the double-negation interpretation already defined in \ref{thedntranslation} for the predicative case; the additional resizing rules are trivially validated in the second point of the Validity Theorem \ref{dnvalidity}.
\end{proof}

We then consider the \textit{classical version of} \cocp\ obtained by adding to its calculus a constant $\mathsf{lem}$ formalising the Law of the Excluded Middle.
\[
\mathsf{lem} \in (\forall x \in \mathsf{Prop})(x \vee \neg\neg x)
\]
We call \cocpc\ the resulting theory. Notice that, contrary to \mf, where we focused on the extensional level to define its classical version, here we chose to add classical logic directly in the intensional level. We think this choice is more in line with the existing literature on classical extensions of the Calculus of (Inductive) Constructions.

\begin{proposition}\label{a7}
\cocpc\ is interpretable in $\mathbf{emTT}_\mathsf{imp}^c+\mathsf{propext}$.
\end{proposition}
\begin{proof}
Thanks to Proposition \ref{a1}, we can refer to the theory $\mathbf{mTT}_\mathsf{imp}$ extended with the constant $\mathsf{lem}$ above. Then, we extend the interpretation of Proposition \ref{eqc1} by sending such new constant to the canonical proof-term of the extensional level $\mathsf{lem} \mapsto \mathsf{true}$. The additional rules assumed on top of those of \mtt, namely the congruence rules, the resizing rules, and the typing axiom of $\mathsf{lem}$ are validated by the interpretation simply because all their translations are equivalent to rules already present in $\mathbf{emTT}^c_\mathsf{imp}$.
\end{proof}

\begin{corollary}\label{coccon}
The theories \cocp\ and \cocpc\ are equiconsistent.
\end{corollary}
\begin{proof}
Following the chain of interpretations depicted below, successively applying Proposition \ref{a7}, Proposition \ref{a5}, Proposition \ref{a6}, and Corollary \ref{a8}.
\[
% https://tikzcd.yichuanshen.de/#N4Igdg9gJgpgziAXAbVABwnAlgFyxMJZABgBpiBdUkANwEMAbAVxiRAEEQBfU9TXfIRQAmclVqMWbAELdeIDNjwEiZAIzj6zVohABhOXyWCia0hupapugCKGF-ZUOSiLE7WwCi3cTCgBzeCJQADMAJwgAWyQyEBwIJGEeUIjoxFE4hMQAZmSQcKikbOp4pAAWPIK0spKstUrUmNqkMxAGLDAdECg6OAALPx8uIA
\begin{tikzcd}
\text{\cocpc} \arrow[d] \arrow[rr, dashed] &             & \text{\cocp}           \\
\mathbf{emTT}_\mathsf{imp}^c+\mathsf{propext} \arrow[r]                    & \mathbf{emTT}_\mathsf{imp}^c \arrow[r] & \mathbf{emTT}_\mathsf{imp} \arrow[u]
\end{tikzcd}
\]
\end{proof}

\section{Conclusions}
We have shown the equiconsistency of the Minimalist Foundation in \cite{m09}, for short \mf, with its classical version.  This is a peculiar property not shared by most foundations for constructive and predicative mathematics, such Martin-L{\"o}f's type theory, Homotopy Type Theory o  Aczel's {\bf CZF}.

In more detail, we have first proved that the levels \mtt\ and \emtt\ of \mf\ are mutually equiconsistent and then that \emtt\ is equiconsistent with its classical version \emttc. 
As a consequence, we have deduced that Dedekind real numbers do not form a set neither in \emttc\  nor in  both levels of \mf. Therefore, \emttc\ can be adopted as a foundation for classical predicative mathematics  à la Weyl, and hence
\mf\ becomes compatible with classical predicativism contrary to most relevant foundations for constructive mathematics.

Finally, we have extended these equiconsistency results  to an impredicative version of \mf\ whose intensional level, called \cocp, coincides with Coquand-Huet's Calculus of Construction in \cite{CoC} extended with  basic inductive type constructor of Martin-Löf's type theory in \cite{MLTT}.
Our contribution extends the equiconsistency result for \coc\ in \cite{CoCprooftheory} with a proof that does not rely  on normalization properties of \cocp.

In the future we intend to exploit a major benefit of  our chain of equiconsistent results, namely that to establish the proof-theoretic strength of \mf, which is still an open problem,
we are no longer bound to refer to \mtt\  but we can interchangeably use \emtt\ or \emttc. A further related goal would be to extend the equiconsistency results presented here  to extensions  of \mf, and of its impredicative version, with
inductive and coinductive definitions investigated in \cite{mmr21,mmr22,MFwtypes}, given
that it is not clear how to extend the Gödel-Gentzen double-negation translation to these extensions.

\paragraph{Acknowledgments}
Fruitful discussions on this work have taken place during the authors' stay at Hausdorff Research Institute for Mathematics for the trimester program ``Prospects of Formal Mathematics''.

\bibliography{main}
\bibliographystyle{plain}

\end{document}